\documentclass[a4paper,11pt]{article}%

\pdfoutput=1 
\usepackage{amsmath}%
\usepackage{amsfonts}%
\usepackage{amssymb}%
\usepackage{amsthm}
\usepackage{stmaryrd}
\usepackage{graphicx}
\usepackage[utf8]{inputenc}
\usepackage[T1]{fontenc}
\usepackage[french, english]{babel}
\usepackage{enumitem}
\usepackage{tikz}
\usetikzlibrary{decorations.markings}
\usepackage[colorlinks=true]{hyperref}

\newtheorem{theorem}{Theorem}[section]

\newtheorem{corollary}[theorem]{Corollary}

\newtheorem{lemma}[theorem]{Lemma}

\newtheorem{proposition}[theorem]{Proposition}

\theoremstyle{definition}
\newtheorem{definition}[theorem]{Definition}

\theoremstyle{remark}

\newtheorem{remark}[theorem]{Remark}

\newlength{\espaceavantspecialthm}
\newlength{\espaceapresspecialthm}
\newcommand{\R}{\mathbb{R}}

{\\}

\newenvironment{defi*}[1][]{
\vskip \espaceavantspecialthm \noindent \textbf{D\'efinition.} }%
{\vskip \espaceapresspecialthm}

\counterwithin{equation}{section}

\tikzset{->-/.style={decoration={
  markings,
  mark=at position .5 with {\arrow{latex}}},postaction={decorate}}}

\newcommand\test[1]{
\pgfmathsetmacro{\var}{#1}
\pgfmathparse{ifthenelse(\var>=0,"positif","négatif")} \pgfmathresult}%

\newcommand{\hgline}[3]{
\pgfmathsetmacro{\thetaone}{#1}
\pgfmathsetmacro{\thetatwo}{#2}
\pgfmathsetmacro{\theta}{(\thetaone+\thetatwo)/2}
\pgfmathsetmacro{\phi}{abs(\thetaone-\thetatwo)/2}
\pgfmathsetmacro{\close}{less(abs(\phi-90),0.0001)}
\ifdim \close pt = 1pt
    \draw[->-, color=#3] (\thetaone:1) -- (\thetatwo:1);
\else
	\pgfmathsetmacro{\R}{tan(\phi)}
	\pgfmathsetmacro{\test}{(\thetaone-\thetatwo)/abs(\thetaone-\thetatwo)}
	\draw[->-, color=#3] (\thetaone:1) arc (\thetaone+\test*90:\thetaone+\test*(270-2*\phi):\R);
\fi
}

\newcommand{\hglinefill}[3]{
\pgfmathsetmacro{\thetaone}{#1}
\pgfmathsetmacro{\thetatwo}{#2}
\pgfmathsetmacro{\theta}{(\thetaone+\thetatwo)/2}
\pgfmathsetmacro{\phi}{abs(\thetaone-\thetatwo)/2}
\pgfmathsetmacro{\close}{less(abs(\phi-90),0.0001)}
\ifdim \close pt = 1pt
    \filldraw[->-, color=#3] (\thetaone:1) -- (\thetatwo:1) ;
\else
	\pgfmathsetmacro{\R}{tan(\phi)}
	\pgfmathsetmacro{\test}{(\thetaone-\thetatwo)/abs(\thetaone-\thetatwo)}
	\filldraw[color=#3, opacity=.2] (\thetaone:1) arc (\thetaone+\test*90:\thetaone+\test*(270-2*\phi):\R) arc (\thetaone+\test*(270-2*\phi)-90:\thetaone+\test*90+90:1);
\fi
}

\begin{document}

\sloppy

\title{Distortion in the group of locally monotone homeomorphisms of a Cantor set and in the group of generalized interval exchange transformations}
\author{Nancy Guelman, Emmanuel Militon}
\maketitle

\begin{abstract}
Let $f$ be either a generalized interval exchange transformation or a locally monotone homeomorphism of a Cantor subset of the real line. In this article, we prove that the following are equivalent.
\begin{enumerate}
    \item The number of discontinuities $d(f^{n})$ of $f^{n}$ is bounded.
    \item There exists $n \geq 1$ such that the element $f$ is conjugate to the restriction to a closed invariant subset of a disjoint union $\mathbb{S}^1_n$ of $n$ circles of a homeomorphism of $\mathbb{S}^1_n$.
    \item The element $f$ is distorted in the group of generalized interval exchange transformations or in the group of locally monotone homeomorphisms of the Cantor subset.
\end{enumerate}
\end{abstract}

\selectlanguage{english}

\section{Introduction}

Let $K$ be a Cantor subset of the circle $\mathbb{S}^1$. A diffeomorphism of $K$ is a homeomorphism of $K$ which locally coincides with a diffeomorphism between two open intervals of the circle. In \cite{FN}, Funar and Neretin introduced the group of such diffeomorphisms and proved that, if the Cantor subset is the standard ternary Cantor subset, the group we obtain contains Thompson's $V_2$ and is very close to being $V_2$. Hence, the authors considered such groups as generalizations of Thompson's groups. Such groups were further studied in the articles by Malicet and the second author. In those articles, some results which can be reminiscent of the Tits alternative for linear groups were proved, namely, in \cite{MM1}, it is proved that subgroups of the group of diffeomorphisms of $K$ without free subsemigroups are metabelian and, in \cite{MM2}, it is proved that any subgroup which does not contain an abelian free subgroup has to preserve a probability measure on the Cantor subset $K$. This last result is valid more generally for the group $\mathrm{Homeo}_{<}(K)$ of homeomorphisms of $K$ which are locally monotone, which is the group under consideration in this article.

A closely related group is the group $\mathrm{GIET}$ of generalized interval exchange transformations, which consists of all piecewise continuous bijections of $[0,1]$, up to a bijection supported on a finite subset. Indeed, take a finitely generated subgroup $G$ of the group of generalized interval exchange transformation. By blowing up the orbits of the discontinuity points of a generating set of $G$, we obtain a locally monotone homeomorphism of a closed subset of the real line. In the case that the group acts minimally on $[0,1]$, the closed subset we obtain is a Cantor subset. Hence it is not surprising that results about the groups $\mathrm{Homeo}_<(K)$ also apply to the group of generalized interval exchange transformations. This article is no exception and the results we prove for the group $\mathrm{Homeo}_<(K)$ translate into similar results for the group $\mathrm{GIET}$.

More specifically, in this article, we describe which elements of $\mathrm{Homeo}_<(K)$ are distorted, which means roughly that the morphism from $\mathbb{Z}$ to $\mathrm{Homeo}_<(K)$ that they induce is not a quasi-isometric embedding in $\mathrm{Homeo}_<(K)$ (see Section \ref{Sec:statements} for the precise definition of distortion element that we use in this article). The notion of distortion elements in groups is a geometric group theory notion introduced by Gromov. In relation to the study of the dynamics of group actions on low-dimensional spaces, distortion has been studied more systematically in transformation groups in recent years (see \cite{Pol}, \cite{FH}, \cite{CF} for instance).  

Indeed, if a group contains no nontrivial (i.e. nontorsion) distorted elements, then, for instance, its torsionfree nilpotent subgroups have to be abelian, which prevents some higher-rank lattices, as finite index subgroups of $\mathrm{SL}_n(\mathbb{Z})$, for $n \geq 3$, from embedding in the group. Closer to the group we are interested in in this article, Novak proved in \cite {Nov} that, in the group of interval exchange transformations, distorted elements are finite order elements. In \cite{GL}, the first author and Liousse proved that in the group of affine interval exchange transformations, if an element is distorted, then its sequence of the number of discontinuity of its iterates is bounded and the element has to be a multirotation.

In this article, we prove an analogous theorem for generalized interval exchange transofrmation, though a bit stronger as we provide a characterization of the distortion elements in $\mathrm{Homeo}_<(K)$ and in $\mathrm{GIET}$. We prove that the distorted elements in those groups are the elements such that the number of discontinuity of its iterates is bounded (see Theorem \ref{thm:equivalence} and Theorem \ref{thm:equivalencegiet} for precise statements). We also prove that the latter condition is equivalent to being a multihomeomorphism of the circle (see Definition \ref{def:multihomeomorphism}). 

This last equivalence provides a dynamical system result which is interesting in itself. It generalizes theorems by Li about interval exchange transformations whose iterates have a bounded number of discontinuities -they are the multirotations- (see \cite{Li}) and by Minakawa about piecewise-linear homeomorphisms of the closed interval with a bounded number of discontinuities and discontinuities of the derivative (see \cite{Min}).

\subsection*{Acknowledgement}

The authors would like to thank the International Research Laboratory IRL-2030 Instituto Franco-Uruguayo de Matemática e Interacciones (IFUMI) in which this article was prepared.

\section{Statement of the results} \label{Sec:statements}

Let $K$ be a Cantor subset of $\mathbb{S}^1$, that is an infinite closed totally disconnected compact subset such that any point is an accumulation point.

\begin{definition}
A locally monotone homeomorphism of $K$ is a homeomorphism $f$ of $K$ such that, for any point $x$ of $K$, there exist open intervals $I$ and $J$ of $\mathbb{S}^1$ which contain respectively $x$ and $f(x)$ as well as a homeomorphism $\tilde{f}:I \rightarrow J$ such that $f_{|I \cap K}=\tilde{f}_{I \cap K}$. 
\end{definition}

Notice that, equivalently, a homeomorphism of $K$ is locally monotone if, for any point of $x$ of $K$, there exist open intervals $I$ and $J$ of $\mathbb{S}^1$ which contain respectively $x$ and $f(x)$, such that $f_{|I \cap K}$ is monotone and sends $I \cap K$ to $J \cap K$. 

We denote by $\mathrm{Homeo}_< (K)$ the group of locally monotone homeomorphisms of $K$.

\begin{definition}
Let $f \in \mathrm{Homeo}_<(K)$. We call discontinuity of $f$ any point $a$ of $K$ which bounds a connected component $I$ of $\mathbb{S}^1 \setminus K$ such that $f(\min I)$ and $f(\max I)$ do not bound a connected component of $\mathbb{S}^1 \setminus K$. We call continuity point of $f$ any point $a$ of $K$ which is not a discontinuity point of $f$.
\end{definition}

In the course of the article, we will need to generalize this last definition to the case of a Cantor set embedded in a disjoint union of circles. Observe also that, given an interval $I$ of $\mathbb{S}^1$ which does not contain any discontinuity point of $f$ in its interior, the restriction of $f$ to $I \cap K$ extends continuously to a homeomorphism between $I$ and another interval of $\mathbb{S}^1$.  

We denote by $\mathrm{Disc}(f)$ the subset of $K$ consisting of discontinuities of $f$. By compactness of $K$ and local monotonicity, the subset $\mathrm{Disc}(f)$ is finite and we denote by $d(f)$ the cardinality of $\mathrm{Disc}(f)$.

In this article, we are interested in characterizing the distortion elements of the group $\mathrm{Homeo}_{<}(K)$. Let us recall now the notion of distortion element.

For a finitely generated group $G$ with finite generating set $S$, we denote by $\ell_S$ the wordlength with respect to $S$. We say that an element $f$ of $G$ is distorted if 
$$ \lim_{n \rightarrow +\infty} \frac{\ell_S(f^n)}{n}=0.$$
Observe that the above limit always exists and that the notion of distortion element is independent of the chosen generating set. Hence the notion of distortion element can be extended to an element of any group by saying that an element is distorted if and only if it belongs to a finitely generated subgroup in which it is distorted.

To state our result, we need the following definition. For $n \geq 1$, we denote by $\mathbb{S}^1_n$ the disjoint union of $n$ circles. 

\begin{definition} \label{def:multihomeomorphism}
We say that an element $f$ of $\mathrm{Homeo}_{<}(K)$ is a multihomeomorphism of the circle if there exist $n \geq 1$ and a continuous locally monotone map $h:K \rightarrow \mathbb{S}^1_{n}$ which is a homeomorphism onto its image so that $hfh^{-1}$ is the restriction to $h(K)$ of a homeomorphism of $\mathbb{S}^1_{n}$.
\end{definition}

The following theorem states that the element which are distorted are exactly those multihomeomorphisms of the circle.

\begin{theorem} \label{thm:equivalence}
Let $f$ be an element of $\mathrm{Homeo}_{<}(K)$. The following are equivalent.
\begin{enumerate}
\item The element $f$ is distorted in $\mathrm{Homeo}_{<}(K)$;
\item the sequence $(d(f^n))_{n \geq 0}$ is bounded;
\item the element $f$ is a multihomeomorphism of the circle.
\end{enumerate}
\end{theorem}

In the next section, we prove the implication $1. \Rightarrow 2.$ as well as other useful properties related to discontinuity points. Sections \ref{s:conjugacy} and \ref{s:distortion} are respectively devoted to the proofs of the implications $2. \Rightarrow 3.$ and $3. \Rightarrow 1.$.

With the same techniques, we can prove an analogous theorem for the group of generalized exchange transformations, that we define now.

\begin{definition}
We call generalized interval exchange transformation any map which is defined on the complement of a finite subset $F$ of $[0,1]$ with values in $[0,1]$, which is continuous on each connected component of $[0,1] \setminus F$ and which is bijective from $[0,1] \setminus F$ onto the complement of a finite subset of $[0,1]$.
\end{definition}

Observe that, as such a map is either increasing or decreasing on each connected component of the complement of $F$, such a map admits a limit on the right and a limit on the left at any point of $[0,1]$. Moreover, a composition of two generalized interval exchange transformations is well defined on the complement of a finite subset of $|0,1]$ and hence defines a generalized interval exchange transformation. Finally, the inverse of a generalized exchange transformation is a generalized exchange transformation.

We want to identify two generalized interval exchange transformations which coincide outside a finite subset. For two given generalized exchanges $f$ and $g$, we write $f \sim g$ if $f$ and $g$ coincide outside a finite subset of $[0,1]$.

\begin{definition}
We call group of generalized interval exchange transformation and we denote by $\mathrm{GIET}$, the quotient of the set of generalized interval exchange transformations endowed with the composition of maps by the equivalence relation $\sim$.
\end{definition}

Using the fact that any generalized interval exchange transformation admits a limit on the left and a limit on the right at any point, the group of interval exchange transformations can be seen as a subgroup of the group of bijections of $[0,1] \times \left\{ -1,1 \right\}$. Indeed, to any element $f$ of $\mathrm{GIET}$, we can associate the bijection $\varphi(f)$ which is defined in the following way.
To a point $(a,-1)$, we associate $(\lim_{a^{-}} f,-1)$ if $f$ is increasing in a left neighborhood of $a$ and $(\lim_{a^{-}} f,1)$ if $f$ is decreasing in a left neighborhood of $a$.
To a point $(a,1)$, we associate $(\lim_{a^{+}} f,1)$ if $f$ is increasing in a left neighborhood of $a$ and $(\lim_{a^{+}} f,-1)$ if $f$ is decreasing in a left neighborhood of $a$.

Fo $f$ in $\mathrm{GIET}$, we call discontinuity of $f$ any point $a$ of $[0,1]$ for which 
$$ \lim_{a^{-}}f \neq \lim_{a^{+}}f.$$
Observe that the set of discontinuity of $f$ does not depend on the chosen representative of $f$ for the equivalence relation $\sim$. Hence, for $f \in \mathrm{GIET}$, we can define $d(f)$ as the cardinal of the set of discontinuities of $F$.

To state a theorem analogous to Theorem \ref{thm:equivalence} which is adapted to this context, we also need to define the relevant notion of multihomeomorphism in this context.

\begin{definition}
We say that an element $f$ of $\mathrm{GIET}$ is a multihomeomorphism of the circle if there exist $n \geq 1$, a map $h$ which is continuous bijection from the complement of a finite subset of $[0,1]$ onto the complement of a finite subset of $\mathbb{S}^1_{n}$, and a homeomorphism $g:\mathbb{S}^1_{n} \rightarrow \mathbb{S}^1_{n}$ such that
$$hfh^{-1}=g$$
outside a finite subset of $\mathbb{S}^1_{n}$.
\end{definition}

Using techniques identical than those used to prove Theorem \ref{thm:equivalence}, we can prove the following.

\begin{theorem} \label{thm:equivalencegiet}
Let $f$ be an element of $\mathrm{GIET}$. The following are equivalent.
\begin{enumerate}
\item The element $f$ is distorted in $\mathrm{GIET}$;
\item the sequence $(d(f^n))_{n \geq 0}$ is bounded;
\item the element $f$ is a multihomeomorphism of the circle.
\end{enumerate}
\end{theorem}

\section{Preliminaries on discontinuity points} \label{s:discontinuity}

The following lemma sums up easy and useful properties of continuity and discontinuity points.

\begin{lemma} \label{lem:compositiondisc}
Let $f,g \in \mathrm{Homeo}_{<}(K)$.
\begin{enumerate}
\item If $x$ is a continuity point of $g$ and $g(x)$ is a continuity point of $f$, then $x$ is a continuity point of $f \circ g$;
\item If $x$ is a continuity point of $g$ and $g(x)$ is a discontinuity point of $f$, then $x$ is a discontinuity point of $f \circ g$;
\item If $x$ is a discontinuity point of $g$ and $g(x)$ is a continuity point of $f$, then $x$ is a discontinuity point of $f \circ g$;
\item $\mathrm{Disc}(f \circ g) \subset \mathrm{Disc}(g) \cup g^{-1}(\mathrm{Disc}(f))$;
\item $\mathrm{Disc}(f^{-1})=f(\mathrm{Disc}(f))$.
\end{enumerate}
\end{lemma}

\begin{proof}
The first three points are consequences of the definition of continuity and discontinuity point.
Clearly, 
$$\left(K \setminus \mathrm{Disc}(g)\right) \cap g^{-1} \left( K \setminus \mathrm{Disc}(f) \right) \subset K \setminus \mathrm{Disc}(f \circ g).$$
Taking the complement, we obtain the fourth point.
If a point belongs to $f\left( K \setminus \mathrm{Disc}(f)\right)$, then it is not a discontinuity of $f^{-1}$ so that, taking complements,
$$\mathrm{Disc}(f^{-1}) \subset f\left( \mathrm{Disc}(f) \right).$$
Exchanging the roles of $f$ and $f^{-1}$ in the above inclusion implies that $f \left(\mathrm{Disc}(f)\right) \subset \mathrm{Disc}(f^{-1})$.
 
\end{proof}

We fix an element $f \in \mathrm{Homeo}_{<}(K)$. For a point $a$ of $K$, we denote 
$$\mathcal{O}(a)= \left\{ f^{k}(a) \ | \ k \in \mathbb{Z} \right\} $$
the orbit of $a$ under $f$. For any set $A$, we denote by $|A|$ its cardinality. For a point $a$ whose orbit $\mathcal{O}(a)$ does not contain any discontinuity point, by Lemma \ref{lem:compositiondisc}, the iterates $f^{n}$ do not contain any discontinuity point on $\mathcal{O}(a)$. When the orbit of $a$ is infinite and contains some discontinuity point, we have the following proposition.

\begin{proposition} \label{prop:discorbits}
Let $a$ be a point of $K$ with infinite orbit under $f$ such that $\mathcal{O}(a) \cap \mathrm{Disc}(f) \neq \emptyset$. Then
\begin{enumerate}
\item either
$$\lim_{n \rightarrow +\infty}\frac{|\mathrm{Disc(f^{n})} \cap \mathcal{O}(a)|}{n}=1;$$
\item or there exists points $b_1, b_2, \ldots, b_k$ of $\mathcal{O}(a)$ such that, for any sufficiently large $n$,
$$\mathrm{Disc}(f^{n}) \cap \mathcal{O}(a)=\left\{ b_1,b_2, \ldots,b_k \right\} \cup f^{-n}(\left\{ b_1,b_2, \ldots,b_k \right\}).$$
Moreover, the subset $f^{-n}(\left\{ b_1,b_2, \ldots,b_k \right\})$ is not contained in $\mathrm{Disc}(f^{2n})$.
\end{enumerate}
\end{proposition}

In the second case, the sequence $(|\mathrm{Disc}(f^{n}) \cap \mathcal{O}(a)|)_{n \geq 0}$ is bounded.
\begin{remark} \label{rema:2.3} The subset $\left\{ b_1,b_2, \ldots,b_k \right\}$ is contained in $\mathrm{Disc}(f^{2n})$, since any point of $f^{n}(\left\{ b_1,b_2, \ldots,b_k \right\})$ is a continuity point of $f^n$.
\end{remark}

\begin{proof}

As the set $\mathrm{Disc}(f) \cap \mathcal{O}(a)$ is finite, there exist integers $N_1 \leq N_2$ such that 
$$\mathrm{Disc(f)} \cap \mathcal{O}(a)\subset \left\{ f^{k}(a) \ | \ N_1 \leq k \leq N_2 \right\}$$
and $f^{N_1}(a)$ and $f^{N_2}(a)$ both belong to $\mathrm{Disc}(f)$.

Changing $a$ into $f^{N_1}(a)$ if necessary, we suppose from now on that there exists $N \geq 1$ such that
$$\mathrm{Disc}(f) \cap \mathcal{O}(a)\subset \left\{ f^{k}(a) \ | \ 0 \leq k < N \right\}$$
and $a$ and $f^{N-1}(a)$ both belong to $\mathrm{Disc}(f)$.
Observe that, for any $n \geq 0$, and any $k \geq N$ or $k \leq -n$, the point $f^{k}(a)$ is a continuity point of $f^{n}$ so that $\mathrm{Disc}(f^n) \subset \left\{f^{i}(a) \ |  -n< i < N \right\}$. Hence
$$ |\mathrm{Disc}(f^{n}) \cap \mathcal{O}(a)| \leq n+N-1.$$
The alternative will depend on whether $a$ is a discontinuity point of $f^N$ or not.

Suppose first that $a$ is a discontinuity point of $f^N$. Take $n\geq N$ and $0 \leq i \leq n-N$. Then the point $f^{-i}(a)$ is a discontinuity point of $f^n=f^{n-N-i}f^N f^{i}$ by points 2. and 3. of Lemma \ref{lem:compositiondisc} as $f^{-i}(a)$ is a continuity point of $f^{i}$ and $f^{N}(a)$ is a continuity point of $f^{n-N-i}$. Hence

$$|\mathrm{Disc}(f^{n}) \cap \mathcal{O}(a)| \geq |\left\{f^{-i}(a) \ | 0 \leq i \leq n-N \right\}|=n-N+1$$
so that 
$$\lim_{n \rightarrow +\infty}\frac{|\mathrm{Disc(f^{n})} \cap \mathcal{O}(a)|}{n}=1.$$

Suppose now that the point $a$ is a continuity point of $f^N$. Fix $n \geq N$. 

For $N-n \leq i\leq 0$, as $f^{i}(a)$ is a continuity point of $f^{-i}$ and $a$ is a continuity point of $f^{n+i}=f^{n-N+i} \circ f^{N}$, then $f^{i}(a)$ is a continuity point of $f^{n}=f^{n+i}\circ f^{-i}$, by the first point of Lemma \ref{lem:compositiondisc}.

For $0<i<N$, by points 1. and 3. of Lemma \ref{lem:compositiondisc}, the point $f^{i}(a)$ is a continuity point of $f^{n}=f^{n-N+i}\circ f^{N-i}$ if and only if it is a continuity point of $f^{N-i}$, as $f^{N}(a)$ is a continuity point of $f^{n-N+i}$.

For $-n<i < N-n$, by points 1. and 2. of Lemma \ref{lem:compositiondisc}, the point $f^{i}(a)$ is a continuity point of $f^{n}=f^{n+i}f^{-i}$ if and only if $a$ is a continuity point of $f^{n+i}$.

If $j=n+i$, the latter condition is equivalent to saying that the point $f^{j}(a)$ is a continuity point of $f^{-j}$ or that the point $f^{j}(a)$ is a continuity point of $f^{N-j}$, as $a$ is a continuity point of $f^{N}$. Hence $f^{i}(a)$ is a continuity point of $f^{n}$ if and only if $f^{i+n}(a)$ is a continuity point of $f^{N-(n+i)}$. By the previous case it is equivalent to  $f^{i+n}(a)$ is a continuity point of $f^{n}$.

Therefore, there exist points $b_1, b_2, \ldots, b_k$ of $\mathcal{O}(a)$ such that, 
$$\mathrm{Disc}(f^{n}) \cap \mathcal{O}(a)=\left\{ b_1,b_2, \ldots,b_k \right\} \cup f^{-n}(\left\{ b_1,b_2, \ldots,b_k \right\}).$$
Analogously, the points of discontinuity of $f^{2n}$ in the orbit of $a$ are $\left\{ b_1,b_2, \ldots,b_k \right\} \cup f^{-2n}(\left\{ b_1,b_2, \ldots,b_k \right\}),$ so the subset $f^{-n}(\left\{ b_1,b_2, \ldots,b_k \right\})$ consists of $f^{2n}$-continuity points.

\end{proof}

As there are finitely many orbits of $f$ which contain a discontinuity of $f$, we deduce immediately the following corollary from Proposition \ref{prop:discorbits}.

\begin{corollary} \label{cor:discalternative}
Let $f \in \mathrm{Homeo}_{<}(K)$. Then
\begin{enumerate}
\item either the sequence $(d(f^n))_n$ is bounded;
\item or the sequence $(d(f^n)/n)_n$ has a nonzero limit.
\end{enumerate}
\end{corollary}

\begin{proof}[Proof of the implication $1. \Rightarrow 2.$ in Theorem \ref{thm:equivalence}]
Suppose $f$ is distorted in $\mathrm{Homeo}_{<}(K)$. Then there exists a finite symmetric subset $S$ of $\mathrm{Homeo}_{<}(K)$ such that $f$ belongs to the subgroup generated by $S$ and
$$ \lim_{n \rightarrow +\infty}\frac{\ell_S(f^{n})}{n}=0.$$
For any $n\geq 0$, decompose $f$ as a product of elements of $S$,
$$f^{n}=s_{1,n} \ldots s_{k_n,n},$$
where $k_n=\ell_S(f^n)$ and $s_{i,j} \in S$. Set $\displaystyle M= \max_{s \in S} d(s)$. By point 4. of Lemma \ref{lem:compositiondisc}, for any $n \geq 0$, 
$$d(f^{n}) \leq \sum_{i=1}^{k_n} d(s_{i,n}) \leq k_n M.$$
so that $\displaystyle \lim_{n \rightarrow +\infty} \frac{d(f^n)}{n}=0$ and Corollary \ref{cor:discalternative} allows to conclude.
\end{proof}

We finish this section with the following classical lemma which will be useful. Fix a family $(I_{1,k})_{1\leq k \leq n}$ of pairwise disjoint clopen intervals of $K$ so that the union of any two intervals of the family is not an interval and so that we encounter in this order $I_{1,1}, I_{1,2},\ldots,I_{1,k}$ when we follow a given orientation of the circle. Fix a second such family $(I_{2,k})_{1\leq k \leq n}$.

\begin{lemma} \label{lem:existhomeo}
There exists an element $f$ of $\mathrm{Homeo}_{<}(K)$ with no discontinuity points such that, for any $k$, $f(I_{1,k})=I_{2,k}$. 
\end{lemma}

\section{Construction of the semiconjugacy} \label{s:conjugacy}

In this section, we prove the implication $2. \Rightarrow 3.$ in Theorem \ref{thm:equivalence}. We suppose that the sequence $(|\mathrm{Disc}(f^n)|)_n$ is bounded.

Denote by $\Sigma$ the finite subset of $K$ which is the union of the periodic orbits of $f$ which contain a discontinuity point of $f$. By Proposition \ref{prop:discorbits}, there exists a finite set $F$, consisting of points with infinite orbits under $f$ such that, for any sufficiently large $n$,
$$\mathrm{Disc}(f^n) \subset F \cup f^{-n}(F) \cup \Sigma.$$

We will divide the proof in three steps. In a first step, we will prove that there exists a power of $f$ which admits a conjugate in $\mathrm{Homeo}_{<}(K)$ with at most two periodic points which are discontinuity points. Moreover, those points are fixed.

Then we will use that to prove that a power of $f$ admits a partition in invariant subsets on which the dynamics is semi-conjugated to that of a circle homeomorphism.

Then we will use this last semi-conjugacy to prove that $f$  itself is a multihomeomorphism of the circle.

Each of the above steps is the object of one of the following subsections.

\subsection{Removing the periodic discontinuity points}

\begin{proposition} \label{prop:removingperiodic}
Let $f$ be an element of $\mathrm{Homeo}_<(K)$ such that the sequence $(d(f^n))_{n \geq 0}$ is bounded. Then there exists a homeomorphism $h$ in $\mathrm{Homeo}_<(K)$ and an integer $\ell >0$, such that, if we denote $g=h^{-1}f^{\ell}h$, one of the following properties holds.
\begin{enumerate}
\item Either there exists a decomposition of $K$ in two $g$-invariant clopen subsets $K_1$ and $K_2$, where $K_i$ is the intersection of an interval with $K$, with the following properties.
\begin{enumerate}
    \item The subset $K_2$ does not contain any discontinuity point of $g$ which is a periodic orbit of $g$;
    \item the only $g$-periodic discontinuity points in $K_1$ are the endpoints of $K_1$ and they are fixed points of $g$.
\end{enumerate}
\item Or $g$ has no periodic point which is a discontinuity point.
\end{enumerate}
\end{proposition}

\begin{proof}

Let $per(f)$ and $fix(f)$ be respectively the set of periodic  and  fixed points of $f$. We begin by noting that as $\mathrm{Disc}(f) \cap per(f)$ is a finite set, there exists an arbitrarily large $N \in \mathbb{N}$ such that $per(f^N) \cap \mathrm{Disc}(f^N)$ coincides with the set of fixed points of $f^N$ that are discontinuity points of $f^N$.

We also claim that the number of discontinuity points of $f^{N}$ that are periodic is even for any large enough $N$. Indeed, observe that the number of discontinuity points of $f^N$ is even and that, by Proposition \ref{prop:discorbits}, for any large enough $N$, the number of discontinuity points of $f^N$ with an infinite orbit is also even, so that the number of discontinuity points of $f^N$ that are periodic is also even.

So from now on, replacing $f$ by $f^N$ for large enough $N$ if necessary, we will assume that $per(f) \cap \mathrm{Disc}(f)=fix(f) \cap \mathrm{Disc}(f)$ and that the cardinality of this subset is even.

Let $p \in fix(f) \cap \mathrm{Disc}(f)$. Without loss of generality we can suppose that there exists $a \in K$ such that $ (p,a) \cap K=\emptyset$. Notice that the orbit of $a$ is necessarily infinite.

There exists a point
$q \in fix(f) \cap \mathrm{Disc}(f)$ such that the interval $(p,q)$  contains no element of $fix(f) \cap \mathrm{Disc}(f)$. As the cardinality of $fix(f) \cap \mathrm{Disc}(f)$ is even, $q \neq p$.

\begin{enumerate}
    \item In the case that there exists $\delta >0$ such that  $(q,q+\delta) \cap K =\emptyset$, we define $h:K\to K$ as a flip in $[a,q] \cap K$, which is a decreasing homeomorphism of $[a,q] \cap K$, and as the identity in $K\cap [a,q]^c$. The existence of such a homeomorphism is a consequence of Lemma \ref{lem:existhomeo}. In particular, the homeomorphism $h$ exchanges $a$ and $q$. It is easy to see that there is no fixed point of $h^{-1}fh$ which is a discontinuity point in $(a,q)\cap K$ since there is no fixed point of $f$ which is a discontinuity point of $f$ in $(a,q)\cap K$.

    As $ h(a)=q$,  $a$ and $ p$ are fixed and continuity points of $h^{-1}fh$.
    So the number of fixed and discontinuity points of $hfh^{-1}$ is equal to the number of fixed and discontinuity points of $f$ minus $2$. This number is still even.
    \item In the case that there exists $\delta >0$ such that  $(q-\delta,q) \cap K =\emptyset$, let $a' \in (a,q) \cap K$ such that there is no discontinuity point of $f$ in $(a',q)\cap K$ and let $q'$ be a point of $\mathrm{Disc}(f)$ such that $\mathrm{Disc}(f) \cap (q,q') \cap K= \emptyset$.
    
    \begin{enumerate}
        \item[i] In the case that $q'\neq p$,    
    we define $h:K\to K$ as an involution that exchanges the intervals in $[a,a']$ and $[q,q']$ with $h(a)=q$ and as the identity in the complement. The existence of such a map is ensured by Lemma \ref{lem:existhomeo}.
    
    Again, as $ h(a)=q$ it follows that $a$ and $ p$ are fixed and continuity points of $h^{-1}fh$.
    
    Since there is no fixed point of  discontinuity of $f$ in $(a,a']$ and $h$ is monotone on $(q,q']$ and on $(a,a']$, there is no fixed discontinuity point of $h^{-1}fh$ in $h^{-1}(a,a']=(q,q']$.
    As there is no point of  discontinuity of $f$ in $(q,q')$  and $h$ is monotone on $(q,q')$ and on $(a,a')$,  there is no fixed discontinuity point of $h^{-1}fh$ in $h^{-1}(q,q')=(a,a')$.
    
    In the case where $q'$ were a discontinuity point of $f$ that is fixed, then $ a'=h^{-1}(q')$ would be a fixed and discontinuity point of $h^{-1}fh$.
    Summarizing, the number of fixed and discontinuity points of $hfh^{-1}$ is equal to the number of fixed and discontinuity points of $f$ minus $2$ and is still even.
    \item [ii] In the case $p=q'$, we have that $K$ can be decomposed as the disjoint union of the invariant sets $K_1=[q,p]\cap K$ and $K_2=[a, a'] \cap K$. Hence item (b) holds. 
    \end{enumerate}

\end{enumerate}

By iterating the construction of items 1 and 2(i) a finite number of times and as the number of discontinuity of $f$ is even, we obtain a homeomorphism $g:K\to K$ which is conjugated to $f$ with no fixed point of discontinuity.

\end{proof}

\subsection{Normal form of a power}

In this section, we will prove the following proposition.

\begin{proposition} \label{prop:normalformpower}
Let $f$ be an element of $\mathrm{Homeo}_{<}(K)$ and suppose the following. 
\begin{enumerate}
\item No periodic point of $f$ is a discontinuity point of $f$.
\item The sequence $(d(f^{n}))_n$ is bounded.
\end{enumerate}
Then there exists $p \geq 1$ such that $f^p$ is a multihomeomorphism of the circle.
\end{proposition}

Observe that, in Proposition \ref{prop:removingperiodic}, in the first case, both $g_{|K_1}$ and $g_{|K_2}$ are conjugate to homeomorphisms which do not possess a discontinuity point which is periodic. Hence the combination of Proposition \ref{prop:removingperiodic} and Proposition \ref{prop:normalformpower} implies that, for any homeomorphism $f \in \mathrm{Homeo}_<(K)$ such that the sequence $(d(f^n))$ is bounded, the homeomorphism $f$ admits a power which is conjugate to a multihomeomorphism of the circle.

To carry out the proof, the following definition taken from \cite{GL} will be convenient (see Definition 5.1 in \cite{GL}).

\begin{definition}
An element $g$ of $\mathrm{Homeo}_{<}(K)$ is said to satisfy the pair property if there exist points $\alpha_1, \alpha_2,\ldots, \alpha_{l}$ of $K$, with infinite orbits under $g$ such that 
\begin{enumerate}
\item $\mathrm{Disc}(g)=\left\{ \alpha_i\ | \ 1 \leq i \leq l \right\} \cup \left\{ g(\alpha_i)\ | \ 1 \leq i \leq l \right\}$.
\item The points $\alpha_i$ are not discontinuity points of $g^2$. 
\end{enumerate}
\end{definition}

Observe that, in this case, the points $g(\alpha_i)$ are discontinuity points of $g^2$.

Fix a homeomorphism $f$ in $\mathrm{Homeo}_{<}(K)$ which satisfies the hypothesis of Proposition \ref{prop:normalformpower}. By Proposition \ref{prop:discorbits}, there exists $p \geq 1$ such that $g=f^{p}$ satisfies the pair property. Set
$$\mathrm{Disc}(g)=\left\{ \alpha_i\ | \ 1 \leq i \leq k' \right\} \cup \left\{ g(\alpha_i)\ | \ 1 \leq i \leq k' \right\}$$
and, for any $1 \leq i \leq k'$, let $\omega_i=g(\alpha_i)$.

\begin{lemma} \label{lem:alphasstaytogether}
For any connected component $I$ of $\mathbb{S}^1 \setminus K$, if one extremity of $I$ is some $\alpha_i$ (respectively some $\omega_i$) then the other extremity of $I$ is some $\alpha_j$, with $j \neq i$ (respt. some $\omega_j$, with $j \neq i$). 
\end{lemma}

\begin{proof}
Take a connected component $I$ of $\mathbb{S}^1 \setminus K$ such that one extremity of $I$ is some $\alpha_i$ and denote by $\beta$ the other extremity of $I$, which has to be a discontinuity point of $g$. As $\alpha_i$ is not a discontinuity point of $g^2$, the point $\beta$ cannot be a discontinuity point of $g^2$. Hence the point $\beta$ cannot be one of the $\omega_j$s and has to be some $\alpha_j$. The proof for the points $\omega_i$ is similar, using that those points are discontinuity points for $g^2$.
\end{proof}

By Lemma \ref{lem:alphasstaytogether}, $k'=2k$, with $k \geq 0$. Moreover, renumbering the discontinuity points if necessary we can suppose that, if we follow the circle starting from the point $\alpha_1$ in the direction given by its orientation, we encounter the points $\alpha_1$, $\alpha_2, \ldots, \alpha_{2k}$ in this order and that the intervals $(\alpha_{2i+1}, \alpha_{2i+2})$, for $i=0,..., k-1$, are connected components of the complement of $\mathbb{S}^1 \setminus K$.

We now define an involution $r$ of $\left\{ 1, \ldots, 2k \right\}$ which will help us determine discontinuity points which can be removed by conjugating $g$. 

This map is defined in the following way. Take an index $i$ and look at the point $\omega_i$. As this point is a discontinuity point of $g$ either it is accumulated only on the right by points of $K$ or it is accumulated only on the left by points of $K$. If it is accumulated on the right by points of $K$, look at the first point among the $\omega_j$, $j \neq i$, on the right of $\omega_i$ and denote $r(i)$ the integer such that this point is $\omega_{r(i)}$. Observe that the point $\omega_{r(i)}$ is accumulated on the left by points of $K$, by Lemma \ref{lem:alphasstaytogether}. The definition of $r(i)$ in the case where the point $\omega_i$ is accumulated on the left by points of $K$ is made by exchanging the words "left" and "right" in the preceding definition. Denote by $I_i$ the interval of $K$ with endpoints $\omega_i$ and $\omega_{r(i)}$. It does not contain any of the points $\omega_j$, with $j \neq i, r(i)$. Note that for any $i$, $I_i=I_{r(i)}$.

\begin{lemma} \label{lem:conjdisc}
Suppose that, for some $0 \leq i\leq k-1$, $r(2i+2) \neq 2i+1$. Then there exists a homeomorphism $h\in \mathrm{Homeo}_{<}(K)$ such that 
\begin{enumerate}
    \item the homeomorphism $hgh^{-1}$ satisfies the pair property,
    \item $$\mathrm{Disc}(hgh^{-1}) \subset h \left(\mathrm{Disc}(g)\right)\setminus h \left(\left\{\alpha_{2i+1},\alpha_{2i+2},\omega_{2i+1},\omega_{2i+2} \right\} \right).$$
\end{enumerate}
\end{lemma}

 Observe that, in particular, the number of discontinuity points of $hgh^{-1}$ is strictly smaller than the number of discontinuity points of $g$. Hence, after repeated use of this lemma, we can suppose that, for any $i$, $r(2i+2)=2i+1$.
 
 \begin{proof}
 Fix an index $i$ as in the hypothesis of the lemma and take an index $j$ such that $\omega_{2i+1}$ and $\omega_j$ bound a connected component of $\mathbb{S}^1 \setminus K$. 
 As the point $\alpha_{2i+1}$ is a discontinuity point for $g$, necessarily, $j \neq 2i+2$. We then distinguish two cases, depending on whether $j=r(2i+2)$ or not. 

Suppose first that $j \neq r(2i+2)$. Then the intervals $I_{2i+1}$, $I_{2i+2}$ and $I_{j}$ of the Cantor set $K$ are pairwise disjoint as $r(2i+2)\neq 2i+1 $. Consider then an element $h$ of $\mathrm{Homeo}_<(K)$ which exchanges the intervals $I_{2i+2}$ and $I_j$, exchanging the points $\omega_{2i+2}$ and $\omega_j$ with no discontinuity point in the interiors of those intervals, and which fixes all the other points of $K$. Observe that
$$\mathrm{Disc}(h)=\left\{ \omega_{2i+2},\omega_{r(2i+2)},\omega_{j},\omega_{r(j)} \right\}.$$
By Lemma \ref{lem:compositiondisc},
$$\begin{array}{rcl}
\mathrm{Disc}(hgh^{-1}) & \subset & \mathrm{Disc}(h^{-1}) \cup h(\mathrm{Disc}(g)) \cup hg^{-1}(\mathrm{Disc}(h)) \\
& \subset & h(\mathrm{Disc}(g))
\end{array}$$

as $\mathrm{Disc}(h)\subset \mathrm{Disc}(g)\cap g(\mathrm{Disc}(g))$ and $\mathrm{Disc}(h^{-1})=h(\mathrm{Disc}(h))=\mathrm{Disc}(h)$. Moreover, observe that $hgh^{-1}$ maps the points $h(\alpha_{2i+1})$ and $h(\alpha_{2i+2})$ respectively to the points $\omega_{2i+1}=h(\omega_{2i+1})$ and $\omega_j=h(\omega_{2i+2})$ which bound a connected component of the complement of $K$. Hence the points $h(\alpha_{2i+1})$ and $h(\alpha_{2i+2})$ are not discontinuity points of $hgh^{-1}$. Moreover, as $\alpha_{2i+1}$ is not a discontinuity point of $g^{2}$, the points $g(\omega_{2i+1})$ and $g(\omega_{2i+2})$ bound a connected component of the complement of $K$. Hence the points $\omega_{2i+1}=h(\omega_{2i+1})$ and $\omega_j=h(\omega_{2i+2})$ are not discontinuity points of $hgh^{-1}$, which completes the proof in this first case.

Suppose now that $j = r(2i+2)$. In that case, take for $h$ a homeomorphism in $\mathrm{Homeo}_<(K)$ that preserves $I_j$, has no discontinuity point in the interior of $I_j$ and that exchanges the points $\omega_j$ and $\omega_{2i+2}$. As in the first case, the homeomorphism $hfh^{-1}$ then satisfies the conclusions of Lemma \ref{lem:conjdisc}.
\end{proof}

The following lemma allows to conclude the proof of Proposition \ref{prop:normalformpower}.

\begin{lemma}
Suppose that, for any $0 \leq i \leq k-1$, $r(2i+2)=2i+1$. Then the homeomorphism $g$ is a multihomeomorphism of the circle.
\end{lemma}

\begin{proof}
Observe that, as each point $\alpha_{2i+1}$ is a continuity point of $g^2$, the points $g(\omega_{2i+1})$ and $g(\omega_{2i+2})$ bound a connected component of $\mathbb{S}^1 \setminus K$. If there was no discontinuity point of $g$ in $I_{2i+1}$ different from $\omega_{2i+1}$ and $\omega_{2i+2}$, then we would have $g(I_{2i+1})=K$, which means that $I_{2i+1}=K$ and that $\omega_{2i+1}$ and $\omega_{2i+2}$ bound a connected component of $\mathbb{S}^1 \setminus K$, a contradiction as $\alpha_{2i+1}$ is a discontinuity point of $g$. 

Hence, each interval $I_{2i+1}$ has to contain at least a pair of points $\alpha_{2j+1}$, $\alpha_{2j+2}$, with $j \leq k-1$. As there are exactly $k$ intervals $I_i$, each interval $I_i$ contains exactly two points $\alpha_{2\sigma(i)+1}$ and $\alpha_{2\sigma(i)+2}$, where $\sigma$ is a permutation of $\left\{0,1, \ldots,k-1 \right\}$.

As $g(\alpha_{2\sigma(i)+1})=\omega_{2\sigma(i)+1}$, $g(\alpha_{2\sigma(i)+2})=\omega_{2\sigma(i)+2}$, $g$ has no discontinuity point in $(\min(I_{2i+1}),\alpha_{2\sigma(i)+1})$ and $(\alpha_{2\sigma(i)+2},\max(I_{2i+1}))$ and the points $g(\omega_{2i+1})$ and $g(\omega_{2i+2})$ bound a connected component of $\mathbb{S}^1 \setminus K$, the homeomorphism $g$ sends the interval $I_{2i+1}$ onto the interval $I_{2\sigma(i)+1}$.

Let $h: K \rightarrow \mathbb{S}^1_k$ be a continuous embedding which sends monotonically, for each $0 \leq i \leq k-1$, the interval $I_{2i+1}$ to the $(i+1)$th circle $S_i$ of the disjoint union of circles $\mathbb{S}^1_k$. Observe that $h(\omega_{2i+1})$ and $h(\omega_{2i+2})$ bound a connected component of the complement of $S_i \setminus h(I_{2i+1})$ so that the homeomorphism $hgh^{-1}$ of the Cantor subset $h(K)$ does not have any discontinuity point.
\end{proof}

\subsection{Normal form}

To prove the implication $2. \Rightarrow 3.$ in Theorem \ref{thm:equivalence}, it remains to prove the following proposition.

\begin{proposition} \label{prop:takingroots}
Let $f$ be a homeomorphism in $\mathrm{Homeo}_<(K)$. Suppose that there exists $p \geq 1$ such that $f^{p}$ is a multihomeomorphism of the circle. Then $f$ is a multihomeomorphism of the circle.
\end{proposition}

In order to prove this proposition, we need the following lemma. 

\begin{lemma} \label{lem:discroot}
Suppose $g$ belongs to $\mathrm{Homeo}_<(K')$, where $K'$ is a Cantor subset of $\mathbb{S}^1$. Suppose also that there exists $p \geq 1$ such that $g^{p}$ does not have any discontinuity point. Then all the discontinuity points of $g$ are periodic points of $g$.
\end{lemma}

\begin{proof}
Suppose for a contradiction that there exists a point $x$ with infinite orbit under $g$ which is a discontinuity point of $g$. Changing $x$ into one of its backward iterate, we can suppose that all the points $g^{-k}(x)$, with $k \geq 1$, are continuity points of $g$. But then the point $g^{-p-1}(x)$ is a discontinuity point of $g^p=g \circ g^{p-1}$, a contradiction. 
\end{proof}

\begin{proof}[Proof of Proposition \ref{prop:takingroots}]
Let $h$ be a continuous embedding $K \rightarrow \mathbb{S}^1_{\ell}$, with $\ell \geq 1$ such that the homeomorphism $hf^{p}h^{-1}$ has no discontinuity point and let $g=hfh^{-1}$. 

Let $(I_i)_{1 \leq i \leq n}$ be the partition of $h(K)$ into intervals of $h(K)$, that is, intersections of an interval of a component of $\mathbb{S}^1_\ell$ with $h(K)$, which do not contain any point in the orbit of a point of discontinuity of $g$ in its interior and with the following property. Either its extremities are in the orbit of points of discontinuity of $g$ or it is the intersection of a connected component of $\mathbb{S}^1_\ell$ with $h(K)$. This partition is finite by Lemma \ref{lem:discroot}.

By definition of this partition, $g$ has no discontinuity point in the interior of each $I_i$. Hence $g$ has to send an element of this partition to another element of this partition.

Let $h': h(K) \rightarrow \mathbb{S}^1_{n}$ be an embedding which is monotone on each interval $I_i$ and sends the $I_i$s to pairwise distinct connected components of $\mathbb{S}^1_{n}$. Then $h'gh'^{-1}$ has no discontinuity point. 

\end{proof}

\section{Elements with a bounded number of discontinuities are distorted} \label{s:distortion}

In this section, we prove the implication $3. \Rightarrow 1.$ in Theorem \ref{thm:equivalence}. The main idea is to apply the lemma of the next section which shows that any sequence of commutators of elements of $\mathrm{Homeo}_{<}(K)$ with small support and no discontinuity points is contained in a finitely generated group with an estimate on the wordlength of each element of the sequence.

To apply this lemma, in Paragraph \ref{ss:decomposition}, we show that any multihomeomorphism of the circle in $\mathrm{Homeo}_{<}(K)$ can be written as a product of a uniformly bounded number of commutators with small support. 

Then, in Paragraph \ref{ss:distortion}, we combine the results of the two preceding paragraphs to prove the implication.

\subsection{A central lemma}

Recall that, for an element $f$ of $\mathrm{Homeo}_{<}(K)$, its support $\mathrm{supp}(f)$ is the closure of the complement of the set of fixed points of $f$.

For any two elements $g,h$ of $\mathrm{Homeo}_{<}(K)$, we denote by $[g,h]=ghg^{-1}h^{-1}$ the commutator of those elements.

This lemma is now classical and is similar to lemmas we can find in \cite{CF}, \cite{Avi}, for instance.

\begin{lemma} \label{lem:distseq}
Let $K_1$ be a proper clopen subset of $K$. Let $(g_n)_{n\geq 0}$ and $(h_n)_{n\geq 0}$ be two sequences of homeomorphisms in $\mathrm{Homeo}_{<}(K)$ which are supported in $K_1$ and with no discontinuity points. Then there exists a finite subset $S$ of $\mathrm{Homeo}_{<}(K)$ such that, for any $n \geq 0$, the element $[g_n,h_n]$ belongs to the group $<S>$ generated by $S$ and
$$\ell_{S}([g_n,h_n]) \leq 14n +12.$$
\end{lemma}

\begin{proof}
The generating set $S$ consists in $5$ homeomorphisms $s_1$, $s_2$, $s_3$, $s_4$, $s_5$ that we construct now. We encourage the reader to look at Figure \ref{fig:trick} while reading the proof.

\begin{figure}
    \centering
    \includegraphics[scale=0.9]{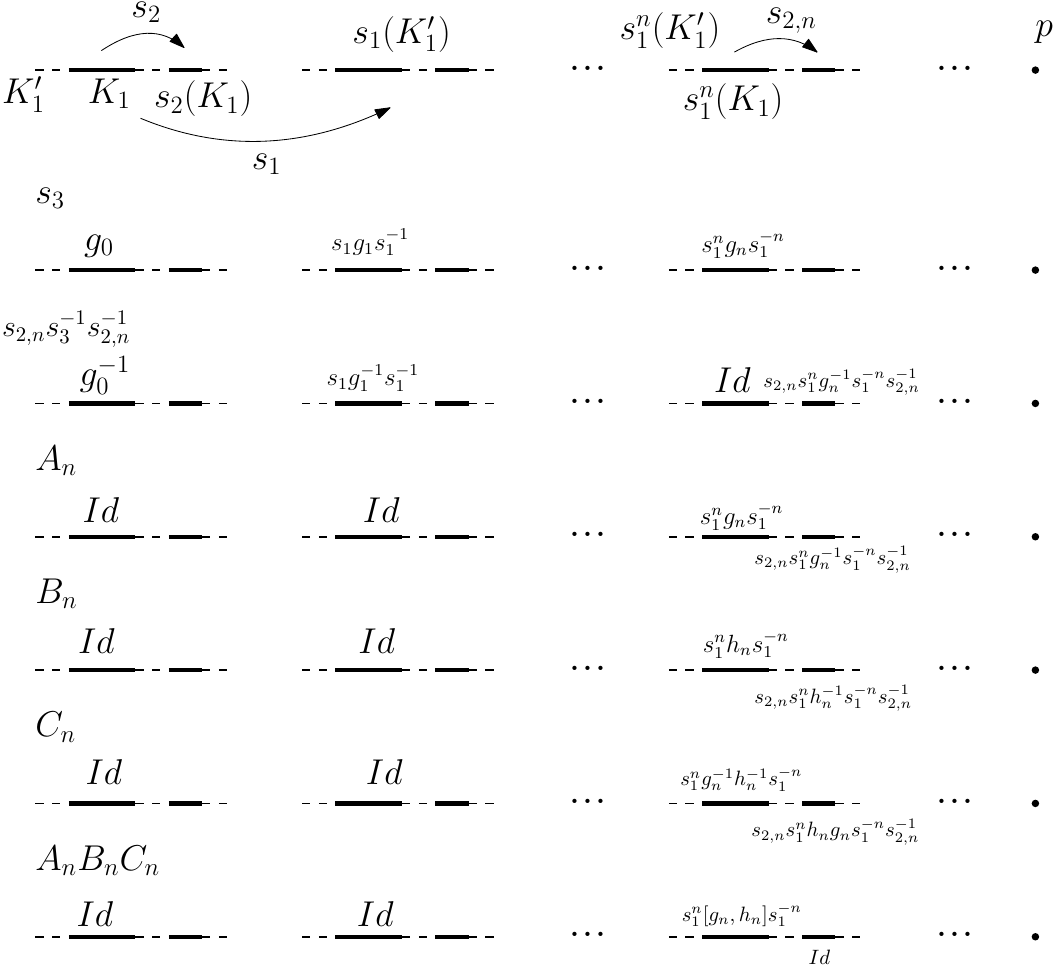}
    \caption{The elements $s_1$, $s_2$, $s_3$, $A_n$, $B_n$ and $C_n$}
    \label{fig:trick}
\end{figure}

Fix a proper clopen subset $K'_1$ of $K$ which contains $K_1$. Take an element $s_1$ of $\mathrm{Homeo}_{<}(K)$ no discontinuity point such that the subsets $s_1^{n}(K'_1)$ are pairwise disjoint and converge to a point $p$. For $s_2$, take any homeomorphism in $\mathrm{Homeo}_{<}(K)$ with no discontinuity point, supported in $K'_1$, which sends $K_1$ to a subset disjoint from $K_1$. Observe that, for $i \geq 0$ the element $s_{2,i}=s_1^i s_2 s_1^{-i}$ is supported in $s_1^i(K'_1)$ and that 
$$s_{2,i}(s_1^i(K_1)) \cap s_1^i(K_1)= \emptyset.$$
Finally, let us define $s_3,s_4,$ and $s_5$ by
$$s_3= \prod_{i=0}^{\infty}s_1^i g_i s_1^{-i}$$
$$s_4= \prod_{i=0}^{\infty}s_1^i h_i s_1^{-i}$$
$$s_5= \prod_{i=0}^{\infty}s_1^i g_i^{-1}h_i^{-1} s_1^{-i}.$$
Observe that those elements are continuous in $p$ and don't have any discontinuity point.

Fix $n\geq 0$ and let $A_n=s_3 s_{2,n}s_3^{-1}s_{2,n}^{-1}$. The element $A_n$ is supported in $s_{1}^{n}(K_1)\cup s_{2,n}(s_{1}^{n}(K_1))$, as $s_{2,n}$ is supported in $s_1^{n}(K'_1)$. It is equal to $s_1^{n}g_n s_1^{-n}$ on $s_{1}^{n}(K_1)$ and to $s_{2,n}s_1^{n}g_n^{-1} s_1^{-n}s_{2,n}^{-1}$ on $s_{2,n}s_{1}^{n}(K_1)$. Observe also that $\ell_{S}(A_n) \leq 4n+4$.

A similar description can be made for $B_n=s_4 s_{2,n}s_4^{-1}s_{2,n}^{-1}$, which is equal to $s_1^{n}h_n s_1^{-n}$ on $s_{1}^{n}(K_1)$ and to $s_{2,n}s_1^{n}h^{-1}_n s_1^{-n}s_{2,n}^{-1}$ on $s_{2,n}s_{1}^{n}(K_1)$ and for $C_n=s_5 s_{2,n}s_5^{-1}s_{2,n}^{-1}$, which is equal to $s_1^{n}g_n^{-1} h_{n}^{-1} s_1^{-n}$ on $s_{1}^{n}(K_1)$ and to $s_{2,n}s_1^{n}h_n g_n s_1^{-n}s_{2,n}^{-1}$ on $s_{2,n}s_{1}^{n}(K_1)$. Observe also that $\ell_{S}(B_n) \leq 4n+4$ and $\ell_{S}(C_n) \leq 4n+4$.

Now the homeomorphism $A_nB_nC_n$ is supported in $s_{1}^{n}(K_1)\cup s_{2,n}(s_{1}^{n}(K_1))$ and is equal to $s_1^{n}g_nh_n g_n^{-1}h_{n}^{-1} s_1^{-n}=s_1^{n}[g_n,h_n] s_1^{-n}$ on $s_{1}^{n}(K_1)$ and to $s_{2,n}s_1^{n}g_n^{-1}h_{n}^{-1}h_ng_n s_1^{-n}s_{2,n}^{-1}=Id$ on $s_{2,n}s_{1}^{n}(K_1)$. Hence
$$A_nB_nC_n=s_1^{n}[g_n,h_n] s_1^{-n}$$
Hence the element $[g_n,h_n]$ belongs to the group generated by $S$ and $\ell_S([g_n,h_n])\leq 14n+12$.

\end{proof}

\subsection{Decomposition of an element as a product of commutators with small support} \label{ss:decomposition}

Let $I_1$ and $I_2$ be two proper clopen intervals of $K$ which cover $K$ such that $I_1 \cap I_2$ is the union of two nonempty disjoint clopen intervals $J_1$ and $J_2$ of $K$ but is not an interval of $K$. For convenience, the indices of the intervals $I_i$ will be taken in $\mathbb{Z}/2\mathbb{Z}$.

\begin{figure}
    \centering
    \includegraphics[scale=0.9]{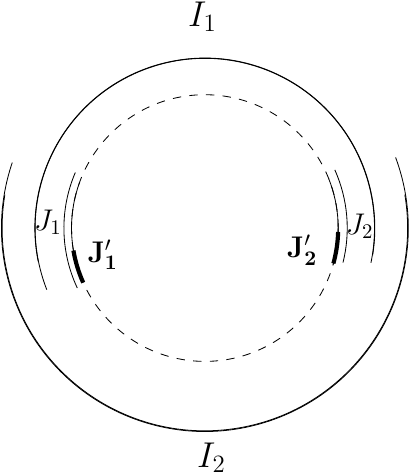}
    \caption{Notation for Lemma \ref{lem:fragmentation}}
    \label{fig:intervals}
\end{figure}

\begin{lemma} \label{lem:fragmentation}
For any homeomorphism $f$ in $\mathrm{Homeo}_<(K)$ which is the restriction to $K$ of an orientation-preserving homeomorphism of the circle, there exist homeomorphisms $f_1$, $f_2$, and $f_3$ in $\mathrm{Homeo}_<(K)$ with no discontinuity point and $i=1$ or $2$ such that the following holds.
\begin{enumerate}
\item the homeomorphisms $f_1$ and $f_3$ are supported in $I_i$ and the homeomorphism $f_2$ is supported in $I_{i+1}$.
\item $ f=f_1f_2f_3.$
\end{enumerate}
\end{lemma}

\begin{proof}
Fix $f$ as in the statement of the lemma. Suppose that $f(I_2 \setminus I_1)$ does not contain $I_1$. Otherwise $f(I_1 \setminus I_2)$ does not contain $I_2$ and exchange the roles of $I_1$ and $I_2$ in the proof below.

The subset $f(I_2 \setminus I_1) \cap I_1$ is contained in the union of at most two clopen intervals of $K$ which are contained in $I_1$ and do not cover $I_1$. Take respective small nonempty clopen subintervals $J'_1$ and $J'_2$ of $J_1$ and $J_2$ such that $(I_2 \setminus I_1) \cup J'_1\cup J'_2$ is an interval of $K$ and $f((I_2 \setminus I_1) \cup J'_1 \cup J'_2)$ does not cover $I_1$.  By Lemma \ref{lem:existhomeo}, there exists an element $g_1$ of $\mathrm{Homeo}_<(K)$ with no discontinuity point which is supported in $I_1$ such that $g_1(f((I_2\setminus I_1) \cup J'_1 \cup J'_2) \cap I_1) \subset J_1 \cup J_2 $ (see Figure \ref{fig:f_1}).

\begin{figure}
    \centering
    \includegraphics[scale=0.9]{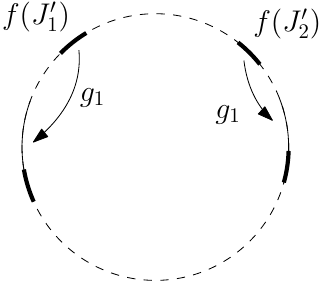}
    \caption{The homeomorphism $g_1$}
    \label{fig:f_1}
\end{figure}

Take respective proper nonempty clopen subintervals $J''_1$ and $J''_2$ of $J'_1$ and $J'_2$ such that $(I_2 \setminus I_1) \cup J''_1\cup J''_2$ is an interval of $K$. Now, take a homeomorphism $f_{2,1}$ in $\mathrm{Homeo}_<(K)$ which is supported in $I_2$, but whose support does not touch the extremities of the interval $I_2$ of $K$, and which is equal to $g_1 f$ on $J''_1 \cup J''_2$ and has no discontinuity point.

Such a homeomorphism exists by Lemma \ref{lem:existhomeo}, as $g_1 f$ is the restriction to $K$ of an orientation-preserving homeomorphism of the circle. As $f_{2,1}^{-1}g_1 f$ has no discontinuity point and is equal to the identity on $J''_1 \cup J''_2$, it preserves the clopen subset $I_2 \setminus I_1$. Let then $f_{2,2}$ be the homeomorphism in $\mathrm{Homeo}_<(K)$ with no discontinuity point which is supported on $I_2 \setminus I_1$ and equal to $f_{2,1}^{-1}g_1 f$ on $I_2 \setminus I_1$. Let $f_2=f_{2,1}f_{2,2}$.

Then it is easy to check that the homeomorphism $f_3=f_2^{-1}g_1f$ is supported in $I_1$. It suffices to take $g_1=f_1^{-1}$ to conclude the proof of the lemma.
\end{proof}

\begin{lemma}
Let $I$ be a proper clopen interval of $K$. For any homeomorphism $f$ of $\mathrm{Homeo}_<(K)$ which is supported in $I$ with no discontinuity point, there exist elements $g$ and $h$ in $\mathrm{Homeo}_<(K)$ with no discontinuity points such that
$$f=[g,h].$$
\end{lemma}

\begin{proof}
Take an element $h$ in $\mathrm{Homeo}_<(K)$, with no discontinuity point, such that the clopen subsets $h^n(I)$, for $n\geq 0$, are pairwise disjoint and converge to a point $p$ of $K$ for the Hausdorff topology. Observe that, for $n\geq 0$, the homeomorphisms $h^nfh^{-n}$ is supported in $h^n(I)$ so that those homeomorphisms have pairwise disjoint supports and pairwise commute. Hence we can define 
$$g=\prod_{n=0}^{+\infty}h^nfh^{-n}.$$
It is the map which is equal to $h^nfh^{-n}$ on the interval $h^n(I)$ and which is equal to the identity elsewhere. Because of the convergence of the sequence of subsets $(h^n(I))_{n \geq 0}$ to a point this map contains no discontinuity point and defines hence a homeomorphism in $\mathrm{Homeo}_<(K)$.
Hence we have
$$\begin{array}{rcl}
[g,h] & = & \displaystyle \prod_{n=0}^{+\infty}h^nfh^{-n}\prod_{n=0}^{+\infty}h^{n+1}f^{-1}h^{-n-1} \\
 & = & \displaystyle \prod_{n=0}^{+\infty}h^nfh^{-n}\prod_{n=1}^{+\infty}h^{n}f^{-1}h^{-n} \\
 & = & f.
 \end{array}$$
\end{proof}

As a direct consequence of the two above lemmas, we obtain the following corollary which will be useful to apply Lemma \ref{lem:distseq}.

\begin{corollary} \label{cor:decompositioncommutators}
For any homeomorphism $f$ in $\mathrm{Homeo}_<(K)$ which is the restriction to $K$ of an orientation preserving homeomorphism of the circle, there exist homeomorphisms $g_1$, $h_1$, $g_2$, $h_2$, $g_3$ and $h_3$ in $\mathrm{Homeo}_<(K)$ with no discontinuity points and $i=1$ or $2$ with the following properties.
\begin{enumerate}
\item the homeomorphisms $g_1$, $h_1$, $g_3$ and $h_3$ are supported in $I_i$ while the homeomorphisms $g_2$ and $h_2$ are supported in $I_{i+1}$.
\item $ f=[g_1,h_1][g_2,h_2][g_3,h_3].$
\end{enumerate}
\end{corollary}

\subsection{End of the proof} \label{ss:distortion}

We prove here the implication $3. \Rightarrow 1.$ in Theorem \ref{thm:equivalence}. Let $f$ be a multihomeomorphism of the circle. By definition, there exists $n \geq 1$ and a locally monotone embedding $h : K \rightarrow \mathbb{S}^1_\ell$ such that $\hat{f}=hfh^{-1}$ has no discontinuity point. For $i=1,\ldots, \ell$, let $K_i$ be the intersection of $h(K)$ with the $i$th circle of $\mathbb{S}^1_\ell$. 

Observe that $f$ is distorted in $\mathrm{Homeo}_{<}(K)$ if and only if some power of $f$ is distorted in $\mathrm{Homeo}_{<}(K)$. Hence, up to replacing $f$ with $f^\ell$, we can suppose that, for any index $i$, $\hat{f}(K_i)=K_i$ and that, for any $i$, $f_i=\hat{f}_{|K_i}$ is the restriction to $K_i$ of an orientation-preserving homeomorphism of the circle. For each Cantor set $K_i$, take two clopen intervals $I_{i,1}$ and $I_{i,2}$ of $K_i$ whose intersection is the union of two clopen intervals. By Corollary \ref{cor:decompositioncommutators}, for any $n\geq 1$, there exist increasing homeomorphisms $g_{i,1,n}$, $h_{i,1,n}$, $g_{i,3,n}$ and $h_{i,3,n}$ in $\mathrm{Homeo}_<(I_{i,1})$ and increasing homeomorphisms $g_{i,2,n}$, $h_{i,2,n}$, $g_{i,4,n}$ and $h_{i,4,n}$ in $\mathrm{Homeo}_<(I_{i,2})$ such that
$$ f_i^{n^2}=[g_{i,1,n},h_{i,1,n}][g_{i,2,n},h_{i,2,n}][g_{i,3,n},h_{i,3,n}][g_{i,4,n},h_{i,4,n}].$$
Moreover, we can take $g_{i,1,n}=h_{i,1,n}=Id$ or $g_{i,4,n}=h_{i,4,n}=Id$ but we will not use this fact. 

By Lemma \ref{lem:distseq}, for any indices $1 \leq i \leq \ell$ and $1 \leq j \leq 4$, there a subset $S_{i,j}$ of elements of $\mathrm{Homeo}_<(h(K))$ supported in $K_i$ such that
\begin{enumerate}
\item for any $n \geq 1$, the commutator $[g_{i,j,n},h_{i,j,n}]$ belongs to the group generated by $S_{i,j}$.
\item $\ell_{S_{i,j}}([g_{i,j,n},h_{i,j,n}]) \leq 14n+12$.
\end{enumerate}.

Let $$\hat{S}=\bigcup_{i,j} S_{i,j}.$$ Then 
$$\ell_{\hat{S}}(\hat{f}^{n^2}) \leq 56n+48.$$

Taking
$$S=\left\{ h^{-1}\hat{s}h | \ \hat{s} \in \hat{S} \right\},$$
we obtain that
$$\ell_{S}(f^{n^2}) \leq 56n+48.$$
so that
$$\lim_{n \rightarrow +\infty} \frac{\ell_{S}(f^{n^2})}{n^2}=0$$
and $f$ is distorted in $\mathrm{Homeo}_<(K)$.

\small
\bibliographystyle{amsalpha}
\bibliography{Biblio}

\end{document}